\documentclass[10pt]{article}

\usepackage{amsmath,amsthm,amssymb,amsfonts}
\usepackage{mathtools}
\usepackage{bm} \usepackage{ulem}
\usepackage{upgreek} \usepackage{pifont} \usepackage{blkarray} \usepackage[all]{xy} \usepackage[shortlabels]{enumitem} \usepackage{subdepth} \usepackage{fouridx} 
\usepackage{todonotes}
\usepackage{graphicx}
\usepackage[mode=image|tex]{standalone}
\usepackage{tikz}
\usepackage{calc}
\usetikzlibrary{calc}
\usetikzlibrary{positioning}
\usetikzlibrary{decorations.pathmorphing,decorations.markings}
\usetikzlibrary{arrows.meta}
\usetikzlibrary{backgrounds}
\usetikzlibrary{fit}

\usepackage{color}

\usepackage{algorithm}
\usepackage{algpseudocode}

\usepackage[colorlinks,linkcolor=red,anchorcolor=blue,citecolor=green]{hyperref}
\usepackage[capitalise,noabbrev]{cleveref}

\newtheorem{theorem}{Theorem}[section]
\newtheorem*{theorem*}{Theorem}

\newtheorem{corollary}{Corollary}[section]
\newtheorem{definition}{Definition}[section]
\newtheorem{example}{Example}[section]
\newtheorem{lemma}{Lemma}[section]
\newtheorem{observation}{Observation}[section]

\newtheorem{remark}{Remark}[section]

\newtheorem{proposition}{Proposition}[section]

\newtheorem{fact}{Fact}[section]

\crefname{claim}{Claim}{Claims}
\crefname{observation}{Observation}{Observations}
\crefname{fact}{Fact}{Facts}

\newcommand{\ZZ}{\mathbb{Z}}
\newcommand{\RR}{\mathbb{R}}
\newcommand{\CC}{\mathbb{C}}

\newcommand{\dd}{\,\mathrm{d}}
\newcommand{\hyperBinom}[2]{\genfrac{[}{]}{0pt}{}{#1}{#2}}

\DeclareMathOperator{\Tr}{Tr}

\DeclareMathOperator{\Hom}{Hom}

\DeclareMathOperator{\diag}{diag}

\DeclareMathOperator{\Span}{span}
\DeclareMathOperator{\End}{End}
\DeclareMathOperator{\supp}{supp}
\DeclareMathOperator{\GL}{GL}

\DeclarePairedDelimiter{\set}{\{}{\}}

\DeclarePairedDelimiter{\abs}{\vert}{\vert}
\DeclarePairedDelimiter{\norm}{\Vert}{\Vert}
\DeclarePairedDelimiter{\lrpar}{(}{)}

\DeclarePairedDelimiter{\lrangle}{\langle}{\rangle}

\begin{document}

\title{Explicit construction of exact unitary designs
}

\author{Eiichi Bannai \footnote{Professor Emeritus of Kyushu University, Fukuoka, Japan.
Postal Address: Asagaya-minami 3-2-33, Suginami-ku, Tokyo 166-0004, Japan. \href{bannai@math.kyushu-u.ac.jp}{bannai@math.kyushu-u.ac.jp}}
        \and
        Yoshifumi Nakata \footnote{
              Photon Science Center, Graduate School of Engineering, The University of Tokyo, Bunkyo-ku, Tokyo 113-8656, Japan.
                JST, PRESTO, 4-1-8 Honcho, Kawaguchi, Saitama, 332-0012, Japan.
              \href{nakata@qi.t.u-tokyo.ac.jp}{nakata@qi.t.u-tokyo.ac.jp}}
        \and
        Takayuki Okuda \footnote{
              Department of Mathematics, Hiroshima University, 1-3-1 Kagamiyama, Higashihiroshima, 739-8526, Japan. 
              \href{okudatak@hiroshima-u.ac.jp}{okudatak@hiroshima-u.ac.jp}}
        \and
        Da Zhao \footnote{School of Mathematical Sciences, Shanghai Jiao Tong University, 800 Dongchuan Road, Minhang District, Shanghai 200240, China. \href{jasonzd@sjtu.edu.cn}{jasonzd@sjtu.edu.cn}}
}

\maketitle

\begin{abstract}
The purpose of this paper is to give explicit constructions of unitary $t$-designs in the unitary group $U(d)$ for all $t$ and $d$. 
It seems that the explicit constructions were so far known only for very special cases. 
Here explicit construction means that the entries of the unitary matrices are given by the values of elementary functions at the root of some given polynomials.
We will discuss what are the best such unitary $4$-designs in $U(4)$ obtained by these methods. 

Indeed we give an inductive construction of designs on compact groups by using Gelfand pairs $(G,K)$. 
Note that $(U(n),U(m) \times U(n-m))$ is a Gelfand pair. 
By using the zonal spherical functions for $(G,K)$, we can construct designs on $G$ from designs on $K$.

We remark that our proofs use the representation theory of compact groups crucially. 
We also remark that this method can be applied to the orthogonal groups $O(d)$, and thus provides another explicit construction of spherical $t$-designs on the $d$ dimensional sphere $S^{d-1}$ by the induction on $d$.

\end{abstract}

\maketitle

\section{Introduction}

The aim of design theory is to approximate a space $M$ by a good finite subset $X$. 
There have been numerous studies on spherical designs\cite{MR0485471} and combinatorial designs\cite{MR2246267}. 
The sphere is a canonical continuous space while the combinatorial design is in the discrete space $M = \binom{V}{k}$ of $k$-subsets of $V$. 
The concept of combinatorial design was generalized to designs on $Q$-polynomial association schemes\cite{MR0384310}. 
Other continuous spaces such as projective space, Grassmannian space\cite{MR2033722,MR1941981,MR2577476} have been considered as well. 
In this paper we focus on the construction of unitary designs, which is designs on the unitary group. 

There is an increasing demand for unitary designs in quantum information science that aims to realize information processing based on quantum mechanics, where protocols are described by unitary transformations. Physically implementing unitary transformations is the key to realize information processing in quantum information science. A unitary transformation chosen uniformly at random from the whole unitary group is of particular importance since it is used in many protocols, such as benchmarking quantum devices~\cite{EAZ2005,KLRetc2008}, characterizing quantum systems~\cite{KZD2016,RBSC2004,MR2433437}, improving computational complexity~\cite{BH2013}, and transmitting information~\cite{HOW2007}. However, the resources available in experiments are quite limited and so, it is in general hard to experimentally implement the uniformly random unitary transformation. Thus, mimicking the whole unitary group by unitary designs is attracting much attention.

While numerous constructions of unitary $2$-designs are known~\cite{CLLW2015,DCEL2009,DJ2011,GAE2007,HL2009,NHMW2015-1}, less is known about unitary $t$-designs for $t \geq 3$~\cite{BHH2016,HMHEGR2020,HL2009TPE,HM2018,NHKW2017}. Most of them are approximate ones. Hence, finding explicit constructions of exact unitary $t$-designs on $U(d)$ is highly desired, which will lead to more accurate realizations of quantum information protocols.
Some finite subgroups of the unitary group are unitary designs of small strength\cite{MR4125850,MR2123127,MR2529619}. The Clifford groups on $U(2^n)$ for any $n \in \mathbb{N}$ are exact unitary $3$-designs\cite{Webb:2016:CGF:3179439.3179447,Zhu_2017,1609.08172}. An exact unitary $4$-design on $U(4)$ is constructed in \cite{Bannai_2019}.

The main purpose of this paper is to provide inductive constructions of exact unitary $t$-design on $U(d)$ for arbitrary strength $t$ and dimension $d$. The method can be applied to orthogonal groups as well. As a by-product we have constructions for complex spherical designs and real spherical designs. 

The existence of spherical designs were proved by Seymour-Zaslavsky\cite{MR744857}. 
Bondarenko-Radchenko-Viazovsk\cite{MR3071504} showed that there exist spherical $t$-designs on $S^{d}$ of size at least $c_d t^d$ for some fixed constant $c_d$ with $t$ tending to infinity. 
This lower bound is asymptotically best possible order of magnitude. 
The asymptotically best lower bound for fixed strength $t$ with $d$ tending to infinity is yet unknown. 
The explicit construction of spherical designs is relatively involved. 
Rabau-Bajnok\cite{MR1133060} and Wagner\cite{MR1089385} constructed spherical $t$-designs from interval $t$-designs with Gegenbauer weight. 
Cui-Xia-Xiang\cite{MR3964155} showed the existence of spherical designs over $\mathbb{Q}(\sqrt{p})$ where $p$ is a prime number. 
Later Xiang\cite{Xiang20+} obtained explicit spherical designs. 
The readers can find surveys on spherical designs in \cite{MR2535394,MR3594369}.

Let us now define unitary designs. 
There are several equivalent definitions of unitary $t$-designs. 

\begin{definition}[{\cite[pp. 14-15]{MR2529619}}]
  Let $X$ be a finite subset $X$ of $U(d)$. 
  The following are equivalent. 
  \begin{enumerate}
    \item $X$ is a unitary $t$-design. 
    \item $\frac{1}{\abs{X}} \sum_{U \in X} U^{\otimes t} \otimes (U^\dagger)^{\otimes t} = \int_{U(d)} U^{\otimes t} \otimes (U^\dagger)^{\otimes t} \dd U$.
    \item $\frac{1}{\abs{X}} \sum_{U \in X} f(U) = \int_{U(d)} f(U) \dd U$ for every $f \in \Hom(U(d),t,t)$, the space of polynomials of homogeneous degree $t$ in entries of $U$ and of homogeneous degree $t$ in the entries of $U^\dagger$.
  \end{enumerate}
\end{definition}

Representation theory is used extensively in our construction. 
There is another equivalent definition of unitary $t$-designs by irreducible representations of $U(d)$, which is quite useful for our purpose. 
The irreducible representations of unitary group are characterized by the highest weight.

\begin{theorem}[{\cite[Theorem 25.5]{MR2062813}}] \label{thm:financiery}
  The irreducible representations of unitary group $U(n)$ are indexed by non-increasing integer sequence $\lambda = (\lambda_1, \lambda_2, \ldots, \lambda_n)$ of length $n$. 
\end{theorem} 

We denote by $\lambda^+$ the sum of positive terms in $\lambda_1, \lambda_2, \ldots, \lambda_n$ and by $\lambda^-$ the absolute value of sum of negative terms in $\lambda_1, \lambda_2, \ldots, \lambda_n$. 
And we define $\abs{\lambda} = \lambda^+ - \lambda^-$. 
The following two collections of irreducible representations are used to characterize unitary design.

\begin{align*}
  \blacksquare_n^{s,t} &:= \set{\lambda: \lambda_1 \geq \lambda_2 \geq \cdots \geq \lambda_n, \lambda^+ \leq s, \lambda^- \leq t}. \\
  \diagup_n^{t} &:= \set{\lambda: \lambda_1 \geq \lambda_2 \geq \cdots \geq \lambda_n, \lambda^+ = \lambda^- \leq t} \subset \blacksquare_n^{t,t}.
\end{align*}

\begin{theorem}[{\cite[Theorem 6]{Bannai_2019}}]
  A finite subset $X \subset U(d)$ is a unitary $t$-design if and only if 
  \begin{equation} \label{eqn:workwoman}
    \frac{1}{\abs{X}} \sum_{U \in X} \rho_\lambda(U) = \int_{U(d)} \rho_\lambda(U) \dd U.
  \end{equation}
  for every irreducible representation $\rho_\lambda$ where $\lambda \in \diagup_n^t$.
\end{theorem} 
  
Due to the inductive nature of our construction, we will replace $\diagup_n^t$ by $\blacksquare_n^{t,t}$. 
We will also adopt the relaxation from set to multi-set for technical reasons. 

\begin{definition}
    Let $X$ be a finite multi-set on $U(d)$. 
    The following are equivalent. 
    \begin{enumerate}
        \item $X$ is a strong unitary $t$-design on $U(d)$.
        \item $\frac{1}{\abs{X}} \sum_{U \in X} U^{\otimes r} \otimes (U^\dagger)^{\otimes s} = \int_{U(d)} U^{\otimes r} \otimes (U^\dagger)^{\otimes s} \dd U$ for every integers $0 \leq r,s \leq t$.
        \item $\frac{1}{\abs{X}} \sum_{U \in X} f(U) = \int_{U(d)} f(U) \dd U$ for every $f \in \Hom(U(d),r,s)$, the space of polynomials of homogeneous degree $r$ in entries of $U$ and of homogeneous degree $s$ in the entries of $U^\dagger$, for every integers $0 \leq r,s \leq t$. 
        \item $\frac{1}{\abs{X}} \sum_{U \in X} \rho_\lambda(U) = \int_{U(d)} \rho_\lambda(U) \dd U$ for every $\lambda \in \blacksquare_n^{t,t}$.
    \end{enumerate}
\end{definition}

This paper shows the following theorem. 

\begin{theorem*}
    Strong unitary $t$-designs on $U(n)$ can be constructed from strong unitary $t$-designs on $U(m)$ and strong unitary $t$-designs on $U(n-m)$ using the zeroes of zonal spherical functions of the complex Grassmannian $\mathcal{G}_{m,n}$.
\end{theorem*}

The paper is organized as follows. 
In \cref{subsection:multi-sets,subsection:repre_and_duals,subsection:Haar_Bochner} the notation for multi-sets, representations and Haar measure are set up. 
In \cref{subsection:ZSF} we introduce Gelfand pairs and zonal spherical functions. 
The central object in this paper, designs on compact groups, are given in \cref{section:designs}. 
We explain our inductive construction in \cref{section:inductive_construction_abstract}. 
In \cref{sec:Schistosoma} we compare different constructions of unitary $4$-designs on $U(4)$. 
We briefly mention how our method gives spherical designs as a by-product in \cref{sec:Liassic}. 
And finally in \cref{sec:swainsona} we discuss the relation between designs in this paper and those in classic design theory. 
The Appendix include the zonal polynomials together with their zeroes.

\section{finite multi-sets on groups}\label{subsection:multi-sets}

Let $G$ be a group.
We use the terminology of 
``non-empty finite multi-sets $X$ on $G$'' in the following sense:
Let $N \in \ZZ_{\geq 1}$.
Each $\mathfrak{S}_N$-orbit in $G^N$ is said to be 
an $N$-point multi-set on $G$,
where $G^N$ denotes the direct product of $N$-copies of $G$ 
and $\mathfrak{S}_N$ the symmetric group of order $N$ acting on $G^N$ as permutations of coordinates.
For the simplicity, the $\mathfrak{S}_N$-orbit of $(x_1,\dots,x_N) \in G^N$ will be denoted by $\{ x_1,\dots,x_N \}_{\mathrm{mult}}$.
Note that $\{ x_1,x_2,\dots,x_N \}_{\mathrm{mult}} \neq \{ x_2,\dots,x_N \}_{\mathrm{mult}}$ even if $x_1 = x_2$.

We also use the following notation:
\begin{itemize}
\item For each $N$-point 
multi-set $X = \{ x_1,\dots,x_N \}_\mathrm{mult}$ on $G$, 
we put $|X| := N$ and 
\[
\sum_{x \in X} \rho(x) := \sum_{i=1}^N \rho(x_i)
\]
for each map $\rho : G \rightarrow W$ where $W$ is an Abelian group.
\item For each $N$-point multi-set $X = \{ x_1,\dots,x_N \}_\mathrm{mult}$ on $G$ 
and each $g \in G$, 
we define $N$-point multi-sets $gX$ and $Xg$ on $G$ by 
\[
gX := \{ gx_1,\dots,g x_N \}_\mathrm{mult} \text{ and } Xg := \{ x_1 g,\dots,x_N g \}_\mathrm{mult},
\]
respectively.
\item Let $G$ and $H$ be both groups.
For an $N$-point multi-set $X = \{ x_1,\dots,x_N \}_\mathrm{mult}$ on $G$ 
and an $M$-point multi-set $Y = \{ y_1,\dots,y_M \}_\mathrm{mult}$ on $H$, 
we define an $NM$-point multi-set $X \times Y$ on $G \times H$ by 
\[
X \times Y := \{ (x_i,y_j) \mid i = 1,\dots,N, ~ j = 1,\dots,M \}_\mathrm{mult}.
\]
\item For each $N$-point and $M$-point multi-sets $X = \{ x_1,\dots,x_{N} \}_\mathrm{mult}$ and $Y = \{ y_1,\dots,y_M \}_\mathrm{mult}$ on $G$, 
we define an $(N + M)$-point multi-set $X \sqcup Y$ and an $NM$-point multi-set $X \cdot Y$ on $G$ respectively by 
\begin{align*}
X \sqcup Y &:= \{ x_1,\dots,x_N,y_1,\dots,y_M \}_\mathrm{mult} \text{ and } \\
X \cdot Y &:= \{ x_iy_j \mid i=1,\dots,N,j=1,\dots,M \}_\mathrm{mult}.
\end{align*}
It should be noted that $X \sqcup Y = Y \sqcup X$ but $X \cdot Y$ is not needed to be equals to $Y \cdot X$ in general.
\item For an ordered finite family of non-empty finite multi-sets $\{ X_i \}_{i = 1,\dots,l}$ on $G$, 
we define 
the non-empty finite multi-set $\bigsqcup_{i=1}^l X_i$ on $G$ inductively by 
\begin{align*}
\bigsqcup_{i=1}^1 X_i &:= X_1, \\
\bigsqcup_{i=1}^{k+1} X_i &:= \left( \bigsqcup_{i=1}^{k} X_i \right) \sqcup X_{k+1} \quad (\text{ for } k = 1,\dots,l-1).
\end{align*}
The non-empty finite multi-set $\prod_{i=1}^l X_i$ on $G$ is also defined inductively as below:
\begin{align*}
\prod_{i=1}^1 X_i &:= X_1, \\
\prod_{i=1}^{k+1} X_i &:= \left( \prod_{i=1}^{k} X_i \right) \cdot X_{k+1} \quad (\text{ for } k = 1,\dots,l-1).
\end{align*}
Note that ``the associativity law'' holds in a natural sense for such the product.
\end{itemize}

\section{Finite-dimensional representations and unitary duals of compact groups}\label{subsection:repre_and_duals}

Let $G$ be a compact Hausdorff (topological) group.
In this subsection, we fix our terminologies for 
finite-dimensional representations of $G$ and the unitary dual $\widehat{G}$ of $G$.

A pair $(\rho,V)$ is said to be a finite-dimensional (complex) $G$-representation 
if 
$V$ is a finite-dimensional complex vector space equipped with the standard topology and 
$\rho : G \rightarrow \GL_\CC(V)$ is a group homomorphism satisfying that 
the map
\[
G \times V \rightarrow V, ~ (g,v) \mapsto gv
\]
is continuous, 
where $\GL_\CC(V)$ denotes the group of all bijective complex linear transformations on $V$.
A finite-dimensional $G$-representation $(\rho,V)$ is called trivial 
if $\rho(g) = \mathrm{id}_V$ for any $g \in G$.

We say that a $\CC$-subspace $W$ of $V$ is $G$-stable   
if $\rho(g)w \in W$ for any $g \in G$ and $w \in W$.
For a $G$-stable subspace $W$ of $V$, 
$\rho$ induces a $G$-representation 
\[
\tau : G \rightarrow \GL_\CC(W), ~ g \mapsto \tau(g) := (\rho(g))|_{W}.
\]
Such the $G$-representation $(\tau,W)$ is called a subrepresentation of $(\rho,V)$.

A finite-dimensional $G$-representation $(\rho,V)$ is called irreducible if 
$V \neq 0$ and there does not exist $G$-stable subspaces of $V$ except for $0$ or $V$ itself.
It is well-known that any finite-dimensional $G$-representation $(\rho,V)$
is completely reducible,
that is, 
for any $G$-stable subspace $W$ of $V$, 
there exists a $G$-stable subspace $W'$ of $V$ with $V = W \oplus W'$.
In particular, any finite-dimensional $G$-representation can be written as a direct sum of finite number of irreducible $G$-representations.

Let $(\rho,V)$ and $(\tau,W)$ be both finite-dimensional $G$-representations.
A $\CC$-linear map $\theta : V \rightarrow W$ is called $G$-intertwining if $\theta$ is commutative to the $G$-actions, 
that is, $\theta \circ \rho(g) = \tau(g) \circ \theta$ for any $g \in G$.
Two finite-dimensional $G$-representations $(\rho,V)$ and $(\tau,W)$ are said to be equivalent if there exists a bijective $G$-intertwining map $\theta : V \rightarrow W$.

Let us consider the case where a Hermitian inner-product on $V$ is given.
Then a finite-dimensional $G$-representation $(\rho,V)$ is said to be unitary if 
$\rho(g)$ preserves the fixed Hermitian inner-product on $V$ for any $g \in G$.
Two finite-dimensional unitary $G$-representations $(\rho,V)$ and $(\tau,W)$ are said to be unitary equivalent if there exists a bijective $G$-intertwining isometric map $\theta : V \rightarrow W$.

Throughout this paper, 
we shall use the terminology of 
the ``unitary dual'' $\widehat{G}$ of $G$
as a complete set of representatives 
for unitary equivalence classes of 
the collection of all irreducible 
finite-dimensional unitary $G$-representations,
that is, 
$\widehat{G}$ is 
a set of irreducible finite-dimensional unitry $G$-representations 
satisfying the following two conditions:
\begin{description}
\item[Condition (1):] For any irreducible finite-dimensional unitary $G$-representation $(\tau,W)$, there exists $(\rho,V_\rho) \in \widehat{G}$ such that $(\tau,W)$ is unitary equivalent to $(\rho,V_\rho)$. 
\item[Condition (2):] Any distinct two elements in $\widehat{G}$ are never unitary equivalent to each other.
\end{description}
Note that all abstract arguments in this paper does not depend on the choice of $\widehat{G}$.

In the theorey of representations of compact groups, 
the following fundamental facts are well-known:
\begin{itemize}
\item Any finite-dimensional $G$-representation $(\rho,V)$ is unitarizable, 
that is, there exists a Hermitian inner-product on $V$ preserved by $\rho(g)$ for any $g \in G$.
\item Two irreducible unitary $G$-representations are unitary equivalent to each other if and only if they are equivalent to each other.
\end{itemize}

By the facts mentioned above, 
one can consider $\widehat{G}$ 
as a complete set of representatives for
for equivalence classes of 
the collection of finite-dimensional irreducible $G$-representations.

\begin{remark}
In this paper, we only deal with finite-dimensional representations.
It should be noted that every irreducible unitary representation of $G$ is known to be of finite-dimension (see \cite[Section 5.1]{Folland95AbstractHA} for the details).
Therefore, our $\widehat{G}$ can be considered as a 
a complete set of representatives 
for equivalence classes of 
the collection of all irreducible unitary $G$-representations.
\end{remark}

\section{Haar measures on compact groups}\label{subsection:Haar_Bochner}

Let $G$ be a compact Hausdorff group.
It is well-known 
that there uniquely exists a Radon measure $\mu$ on $G$ 
satisfying the following two condisions (see \cite[Sections 2.3 and 2.4]{Folland95AbstractHA} 
for more details):
\begin{description}
\item[Condition (1):] Our $\mu$ is a probability measure, that is, $\mu(G) = 1$.
\item[Condition (2):] Our $\mu$ is two-sided invariant, 
that is, for any Borel set $E$ of $G$ and any $g \in G$
the equality below holds:
\[
\mu(E) = \mu(gE) = \mu(Eg).
\]
\end{description}
Throughout this paper, we call such the measure $\mu$ the probability (two-sided) Haar measure on $G$,
and denote simply by 
\[
\int_G f(g) \dd g \in \CC
\]
the integration of 
a $\CC$-valued $L^1$-integrable function $f$ on $G$ 
with respect to the probability Haar measure $\mu$.
Note that any continuous $\CC$-valued functions on $G$ is $L^1$-integrable with respect to $\mu$.

Let $W$ be a finite-dimensional complex vector space equipped with the standard topology.
For each continuous map $\rho : G \rightarrow W$, 
one can easily see that 
there uniquely exists a vector 
\[
\int_{G} \rho(g) \dd g \in W,
\]
which is called the Bochner integration of $\rho$ in $W$, 
such that for any $\CC$-linear functional $\eta : W \rightarrow \CC$, 
the equality below holds:
\[
\lrangle*{\int_{G} \rho(g) \dd g, \eta} = \int_{G} \langle \rho(g), \eta \rangle \dd g,
\]
where $\langle \cdot, \cdot \rangle$ denotes the pairing on $W \otimes W^\vee$.
Note that $g \mapsto \langle \rho(g), \eta \rangle$ defines a continuous function on $G$
and thus the right hand side is well-defined. 
One may check \cite[Appendix 4]{Folland95AbstractHA} or \cite{MR0492147} for more information on Bochner integral.

In \cref{section:designs}, 
we consider the Bochner integration 
\[
\int_{G} \rho(g) \dd g \in \End_\CC(V)
\]
for each finite-dimensional $G$-representation $(\rho,V)$ 
by considering the continuous map 
\[
\rho : G \rightarrow \GL_\CC(V) \subset \End_\CC(V).
\]
Then the Bochner integration can be characterized as 
a unique operator in $\End_\CC(V)$ with 
\[
\left( \int_{G} \rho(g) \dd g \right) (v) = \int_{G} (\rho(g)v) \dd g
\]
where the right hand side is defined as the Bochner integration of the continuous map 
\[
G \rightarrow V, ~ g \mapsto \rho(g)v.
\]

The following well-known fact will be applied throughout this paper, 
\begin{fact}\label{fact:Bochner_orthogonal_proj}
Let us assume that $(\rho,V)$ is a finite-dimensional unitary $G$-representation.
Then the operator  
\[
A = \int_{G} \rho(g) \dd g \in \End_\CC(V)
\]
is the orthogonal projection of $V$ onto $V^G := \{ v \in V \mid \rho(g)v = v \}$ since $A^2 = A = A^*$.
In particular, 
\[
\int_{G} (\rho(g)v) \dd g \in V^G
\]
holds for any $v \in V$.
\end{fact}

\section{Gelfand pairs and zonal spherical functions}\label{subsection:ZSF}

Let $G$ be a compact Hausdorff group and $K$ a closed subgroup of $G$.
We say that $(\rho,V_\rho) \in \widehat{G}$ is $K$-spherical if 
\[
V_\rho^K := \{ v \in V_\rho \mid \rho(k)v = v \text{ for any } k \in K \} \neq 0, 
\]
and we let
\[
\widehat{G}_K := \{ (\rho,V_\rho) \in \widehat{G} \mid (\rho,V_\rho) \text{ is } K\text{-spherical} \}.
\]
The pair $(G,K)$ is said to be a compact Gelfand pair if 
\[
\dim_\CC V_\rho^{K} = 1
\]
for any $(\rho,V_\rho) \in \widehat{G}_K$.

\begin{example}[{\cite[Proposition 8.1.3]{MR2328043} and \cite[Theorem 5.1]{TakeuchiModern} }]
Any compact symmetric pair is a compact Gelfand pair.
\begin{enumerate}
\item $(U(n),U(k) \times U(n-k))$.
\item $(O(n),O(n-1))$.
\item $(G_0 \times G_0,\diag G_0)$.
\end{enumerate}
\end{example}

\begin{example}[{\cite[Example 2.1]{MR882540}}]
Let us assume that $G,K$ are both finite.
Then $(G,K)$ is a Gelfand pair if and only if the Schurian scheme corresponding to $(G,K)$ is commutative as an association scheme.
\end{example}

Let $(G,K)$ be a compact Gelfand pair.
A continuous $\CC$-valued function $f$ on $G$ is said to be bi-$K$-invariant
if $f(k_1 x k_2) = f(x)$ for any $x \in G$ and any $k_1,k_2 \in K$.
For each $(\rho,V_\rho) \in \widehat{G}_K$, 
by taking a unit $K$-invariant vector $v \in V_\rho^K$,
we define the continuous bi-$K$-invariant function 
\[
Z^{(G,K)}_\rho : G \rightarrow \CC, ~ \omega \mapsto \langle v,\rho(\omega)v \rangle_{\rho},
\]
where $\langle \cdot,\cdot \rangle_\rho$ denotes the fixed $G$-invariant Hermitian inner-product on $V_\rho$.
One can easily see that $Z^{(G,K)}_\rho$ does not depend on the choice of the unit $K$-invariant vector $v \in V_\rho^K$.
We call $Z^{(G,K)}_\rho \in C(G,K)$ the normalized zonal spherical function for $(\rho,V_\rho) \in \widehat{G}_K$.
Note that our ``normalization'' means that $Z^{(G,K)}_\rho(e_G) = 1$, where $e_G$ denotes the unit of $G$.

By the definition, one can easily observe the following proposition:

\begin{proposition}\label{proposition:inv_conj}
The equality 
\[
Z^{(G,K)}_\rho(\omega^{-1}) = \overline{Z^{(G,K)}_\rho(\omega)}
\]
holds for any $\omega \in G$ and any $\rho \in \widehat{G}_K$.
\end{proposition}

The proposition below for zeroes of $Z^{(G,K)}_\rho$ in $G$ 
will be needed in \cref{section:inductive_construction_abstract}.

\begin{proposition}\label{proposition:zero_zonal}
Let us fix a non-trivial $K$-spherical irreducible unitary $G$-representation $\rho \in \widehat{G}_K$.
\begin{enumerate}
\item If $Z_\rho^{(G,K)}$ is real valued on $G$, then $Z_\rho^{(G,K)}$ has some zeroes in $G$.
\item Even if $Z_\rho^{(G,K)}$ has no zero point in $G$, 
the real part 
\[
\Re Z_\rho^{(G,K)} : G \rightarrow \RR, ~ \omega \mapsto \Re (Z_\rho^{(G,K)}(\omega)),
\]
has some zeroes in $G$.
\end{enumerate}
\end{proposition}

\begin{proof}[Proof of \cref{proposition:zero_zonal}]
It is well-known that 
\[
\int_{G} Z_\rho^{(G,K)}(g) \dd g = 0
\]
for any non-trivial $\rho \in \widehat{G}_K$
(see \cite[Section 1.1]{TakeuchiModern}).
Therefore, our claim follows from the intermediate value theorem on $\RR$.
\end{proof}

\section{Designs on compact groups}\label{section:designs}

Let $G$ be a compact Hausdorff group equipped with the probability Haar measure.
In this section, we give a definition and study some basic properties of designs on $G$ for finite-dimensional $G$-representations.

\subsection{Definitions and characterizations of designs on compact groups}\label{subsection:definition_designs}

For each finite-dimensional $G$-representaion $\rho$, 
we give a definition of ``$\rho$-designs'' on $G$ as follows:

\begin{definition}
Let $(\rho,V)$ be a finite-dimensional $G$-representation.
A non-empty finite multi-set $X$ on $G$ is said to be a $\rho$-design on $G$ if 
\[
\frac{1}{|X|} \sum_{x \in X} \rho(x) = \int_{G} \rho(g) \dd g 
\]
in the vector space $\End_\CC(V)$ (see \cref{subsection:Haar_Bochner,subsection:multi-sets} for the notation).
Furthermore, for each collection $\Lambda$ of finite-dimensional $G$-representations, 
a non-empty finite multi-set $X$ on $G$ is called a $\Lambda$-design on $G$ 
if $X$ is a $\rho$-design on $G$ for any $\rho \in \Lambda$.
\end{definition}

It should be noted that the concept of $\rho$-designs on $G$
 depends only on the equivalent classes of $\rho$,
that is, 
if $\rho$, $\rho'$ are equivalent as finite-dimensional $G$-representations, 
then the concept of $\rho$-designs on $G$ and that of $\rho'$-designs on $G$ coincide with each other.

We also observe that the concept of designs on $G$ is invariant by natural $(G \times G)$-actions on $G$ in the following sense:
Let $\rho$ be a finite-dimensional representation of $G$,
and $X$ a $\rho$-design on $G$.
Then for any $g \in G$, the non-empty finite multisets $Xg$ and $gX$ are both $\rho$-designs on $G$.
In particular, for any collection $\Lambda$ of finite-dimensional representations of $G$, any $\Lambda$-design $X$ on $G$ and any $g \in G$, 
the non-empty finite multisets $Xg$ and $gX$ are both $\Lambda$-designs on $G$.

One can also easily see that if we have $\rho$-designs $X_1$ and $X_2$ on $G$, 
then $X_1 \sqcup X_2$ is also a $\rho$-design on $G$.
In particular, we obtain the following claim which will be applied in \cref{subsection:ind_const}:

\begin{proposition}\label{proposition:product_of_designs}
Let $X$ be a $\rho$-design on $G$
for a finite-dimensional $G$-representations $\rho$,
and fix non-empty finite multi-sets $Y$ and $Z$ on $G$.
Then the non-empty finite multi-set $Y \cdot X \cdot Z$ on $G$ is also a $\rho$-design on $G$.
\end{proposition}

The following proposition is also needed in \cref{section:inductive_construction_abstract}:

\begin{proposition}\label{proposition:inv_designs}
Let $(\rho,V)$ be a finite-dimensional $G$-representation
and $X$ a non-empty finite multi-set on $G$.
The following two conditions on $X$ are equivalent:
\begin{description}
\item[Condition (i):] $X$ is a $\rho$-design on $G$.
\item[Condition (ii):] $X^{-1}$ is a $\rho$-design on $G$.
\end{description}
\end{proposition}

\begin{proof}[Proof of \cref{proposition:inv_designs}]
Recall that any finite-dimensional $G$-representation is unitarizable.
We take a $G$-invariant Hermitian inner-product $\lrangle{\cdot,\cdot}$ on $V$.
Then the operator 
\[
\int_{G} \rho(g) \dd g
\]
is the orthogonal projection of $V$ onto $V^G$,
and hence it is self-adjoint with respect to $\lrangle{\cdot,\cdot}$.
We also observe that for each $g \in G$, 
the operator $\rho(g^{-1})$ is the adjoint of $\rho(g)$ on $V$
with respect to $\lrangle{\cdot,\cdot}$.
In particular, the operator 
\[
\frac{1}{|X|} \sum_{x \in X} \rho(x^{-1})
\]
is the adjoint of 
\[
\frac{1}{|X|} \sum_{x \in X} \rho(x).
\]
Therefore, 
the equality 
\[
\frac{1}{|X|} \sum_{x \in X} \rho(x) = \int_{G} \rho(g) \dd g
\]
is equivalent to the equality 
\[
\frac{1}{|X|} \sum_{x \in X} \rho(x^{-1}) = \int_{G} \rho(g) \dd g
\]
by taking adjoints. 
\end{proof}

In the viewpoints of relationship between $G$-representions, 
one can easily see that the following holds:
\begin{itemize}
\item Let $\rho$ be a finite-dimensional representation of $G$ and $\tau$ a subrepresentation of $\rho$.
Then any $\rho$-design on $G$ is also a $\tau$-design on $G$.
\item Let $\rho_1$, $\rho_2$ be both finite-dimensional $G$-representations.
Then for non-empty finite subset $X$ of $G$, the following conditions are equivalent:
\begin{enumerate}
\item $X$ is a $\rho_1$-design and a $\rho_2$-design on $G$, simultaneously.
\item $X$ is a $(\rho_1 \oplus \rho_2)$-design on $G$.
\end{enumerate}
\end{itemize}

Let us denote by $\widehat{G}$ the unitary dual of the compact group $G$ (see \cref{subsection:repre_and_duals}).
Recall that any finite-dimensional $G$-representation is unitarizable and completely reducible.
Therefore, for each collection $\Lambda$ of finite-dimensional $G$-representations, by putting 
\[
\supp \Lambda := \bigcup_{\sigma \in \Lambda} \{ \rho \in \widehat{G} \mid \rho \text{ appears as a subrepresentation of } \sigma \} \subset \widehat{G},
\]
the concept of $\Lambda$-designs on $G$ 
and that of $(\supp \Lambda)$-designs on $G$ are equivalent.

Let us fix a irreducible unitary $G$-representation $\rho \in \widehat{G}$ 
and give a characterization of $\rho$-designs on $G$ below.

\begin{proposition}\label{proposition:characterization_of_designs}
Let $(\rho,V_\rho) \in \widehat{G}$.
\begin{enumerate}
\item If $\rho$ is the trivial as a $G$-representation,
then any non-empty finite multi-set on $G$ is a $\rho$-design.
\item Let us consider the case where $\rho$ is not trivial as a representation of $G$.
Then for a non-empty finite multi-set $X$ of $G$, 
the following two conditions are equivalent:
\begin{enumerate}
\item $X$ is a $\rho$-design on $G$.
\item The following equality holds in $\End_\CC(V_\rho)$:
\[
\sum_{x \in X} \rho(x) = 0.
\]
\end{enumerate}
\end{enumerate}
\end{proposition}

\begin{proof}[Proof of \cref{proposition:characterization_of_designs}]
In the case where $\rho$ is trivial, the claim is easy.
Thus we assume that $\rho$ is not trivial as a representation of $G$.
In order to prove our claim, it suffices to show that 
\[
\left(\int_{G} \rho(g) \dd g \right) v = 0 
\]
in $V_\rho$ for any $v \in V_\rho$.
Take any $v \in V_\rho$.
Then by \cref{fact:Bochner_orthogonal_proj}, 
the Bochner integration 
\[
\int_{G} (\rho(g) v) \dd g \in 
V_\rho
\]
is a $G$-invariant vector in $V_\rho$.
Recall that $V_\rho$ is irreducible and non-trivial, 
and hence $V_\rho$ has no non-zero $G$-invariant vectors.
This implies that 
\[
\left( \int_{G} \rho(g) \dd g \right) v = \int_{G} (\rho(g) v) \dd g = 0
\]
in $V_\rho$.
\end{proof}

We also note that in the case where our multi-set on $G$ 
is a finite subgroup of $G$, 
the following proposition holds 
as a direct corollary to \cref{fact:Bochner_orthogonal_proj}:
\begin{proposition}
Let $\Gamma$ be a finite subgroup of $G$ and $(\rho,V)$ a finite-dimensional $G$-representation.
We consider $\Gamma$ as a non-empty finite multi-set on $G$ in a natural sense.
Then the following two conditions on $\Gamma$ are equivalent:
\begin{description}
\item[Condition (i):] $\Gamma$ is a $\rho$-design on $G$. 
\item[Condition (ii):] $V^\Gamma = V^G$, where $V^\Gamma$ [resp.~$V^G$] denotes the subspaces of all $\rho(\Gamma)$ [resp.~$\rho(G)$] fixed vectors in $V$. 
\end{description}
\end{proposition}

\subsection{Designs on direct products of compact groups}\label{subsection:designs_on_products}

Let $G$ and $H$ be both compact Hausdorff groups.
Then the direct product $G \times H$ is also a compact Hausdorff group.
In this subsection, we study designs on $G \times H$.

For a finite-dimensional $G$-represention $(\rho,V)$ 
and a finite-dimensional $H$-representation $(\sigma,W)$,
we obtain a finite-dimensional $(G \times H)$-represention
$(\rho \boxtimes \sigma,V \otimes W)$ defined by 
\[
(\rho \boxtimes \sigma)(g,h) := \rho(g) \otimes \sigma(h) \in \GL_\CC(V \otimes W)
\]
for each $(g,h) \in G \times H$.

The following proposition will be applied for our inductive constructions
of unitary designs:

\begin{proposition}\label{proposition:designs_on_products}
Let $(\rho,V)$ and $(\sigma,W)$ be a finite-dimensional $G$-representation and a finite-dimensional $H$-representation, respectively. 
Take any $\rho$-design $X$ on $G$ and any $\sigma$-design $Y$ on $H$.
Then the non-empty finite multi-set $X \times Y$ on $G \times H$ is
a $(\rho \boxtimes \sigma)$-design on $G \times H$.
\end{proposition}

\begin{proof}[Proof of \cref{proposition:designs_on_products}]
By the Fubini's theorem for Haar measures on $G \times H$, 
one can easily check that 
\[
\int_{G \times H} (\rho(g) \otimes \sigma(h)) \dd (g,h) \\
  = \left( \int_{G} \rho(g) \dd g \right) \otimes \left( \int_{H} \sigma(h) \dd h \right).
\]
Therefore, we obtain that 
\begin{align*}
\int_{G \times H} (\rho \boxtimes \sigma) (g,h) \dd (g,h)
  &= \int_{G \times H} (\rho(g) \otimes \sigma(h)) \dd (g,h) \\
  &= \left( \int_{G} \rho(g) \dd g \right) \otimes \left( \int_{H} \sigma(h) \dd h \right) \\
  &= \left( \frac{1}{|X|}\sum_{x \in X} \rho(x) \right) \otimes \left( \frac{1}{|Y|}\sum_{y \in Y} \sigma(y) \right) \\
  &= \frac{1}{|X \times Y|} \sum_{(x,y) \in X \times Y} \rho(x) \otimes \sigma(y).
\end{align*}
This proves our claim.
\end{proof}

Let us denote by $\widehat{G}$ and $\widehat{H}$ the unitary duals of $G$ and $H$ (in the sense of \cref{subsection:repre_and_duals}), respectively. 
Suppose at least one of $G$ and $H$ is second countable. 
Then it should be noted that  
the set 
\[
\{ \rho \otimes \sigma \mid \rho \in \widehat{G}, \sigma \in \widehat{H} \}
\]
can be considered as a unitry dual of $G \times H$,
that is, 
$\rho \otimes \sigma$ is irreducible as $(G \times H)$-representation
for any $(\rho,\sigma) \in \widehat{G} \times \widehat{H}$, 
and for each finite-dimensional irreducible $(G \times H)$-representation $\pi$, 
there unquely exists $(\rho,\sigma) \in \widehat{G} \times \widehat{H}$ 
such that $\pi$ is equivalent to $\rho \otimes \sigma$ 
(see \cite[Theorem 7.17]{Folland95AbstractHA} for the details).

\section{Inductive constructions of designs on compact groups} \label{section:inductive_construction_abstract}

Let $G$ be a second countable compact Hausdorff group and $(G,K)$ a compact Gelfand pair as in \cref{subsection:ZSF}.
Then $K$ itself is also a compact Hausdorff group.
In this section, we give an algorithm to construct designs on $G$
from designs on $K$ and zeros of zonal spherical functions for $(G,K)$.

\subsection{Inductive constructions for compact groups} \label{subsection:ind_const}

For each finite-dimensional $G$-representation $(\rho,V)$,
we obtain the $K$-representation $(\rho|_K,V)$ by defining 
\[
\rho|_K : K \rightarrow \GL_\CC(V), ~ k \mapsto \rho(k).
\]
Note that even if $\rho$ is irreducible as $G$-representation, 
$\rho|_K$ is not needed to be irreducible as $K$-representation.

In order to state our results simply, 
we use the following symbols: 
\[
\widehat{G}_K^* := \widehat{G} \setminus \{ \text{the trivial irreducible } G\text{-representation} \}.
\]
or
\[
\widehat{G}_K^* := \widehat{G}_K \cap \widehat{G}^*.
\]

One of main results of this paper is the following:

\begin{theorem}\label{theorem:gen_for_irred}
Let us fix $\rho \in \widehat{G}^*$ where $G$ is a second countable compact group.
Take any $(\rho|_K)$-design $Y$ on the compact Hausdorff group $K$.
\begin{enumerate}
\item If $\rho$ is not $K$-spherical, then $Y$ itself is a $\rho$-design on $G$.
\item Let us consider the case where $\rho$ is $K$-spherical. 
Then for any non-empty finite multi-set $\Omega$ on $G$ with 
\[
\sum_{z \in \Omega} Z_\rho^{(G,K)}(z) = 0,
\]
the non-empty finite multi-set $X$ in $G$ defined by 
\[
X := Y \cdot \Omega \cdot Y 
\]
is a $\rho$-design on $G$.
\end{enumerate}
\end{theorem}

\begin{proof}[Proof of \cref{theorem:gen_for_irred}]
Recall that by Proposition \ref{proposition:inv_designs}, 
$Y^{-1}$ is also a $\rho$-design on $G$.
By Fact \ref{fact:Bochner_orthogonal_proj}, 
the operator 
\[
\frac{1}{|Y|} \sum_{y \in Y} \rho(y) = \int_{K} \rho(k) dk = \frac{1}{|Y|} \sum_{y \in Y} \rho(y^{-1}).
\] 
on $V_\rho$
is the orthogonal projection onto $V_\rho^K$.

In the case where $\rho$ is not $K$-spherical, 
$V_\rho^K = 0$ and hence 
\[
\sum_{y \in Y} \rho(y) = 0.
\] 
Therefore, $Y$ itself is a $\rho$-design on $G$ 
by \cref{proposition:characterization_of_designs} in this case.

Let us consider the case where $\rho$ is $K$-spherical.
We fix a non-empty finite multi-set $\Omega$ on $G$ with 
\[
\sum_{z \in \Omega} Z_\rho^{(G,K)}(z) = 0.
\]
By \cref{proposition:characterization_of_designs}, 
we only need to show that 
\[
\sum_{x \in X} \rho(x) = 0
\]
in $\End(V_\rho)$
for $X := Y \cdot \Omega \cdot Y$.
Let us take any vector $v_1,v_2 \in V_\rho$.
It suffices to show that 
\[
\lrangle*{v_1, \sum_{x \in X} \rho(x) v_2}_\rho = 0
\]
where $\lrangle{\cdot,\cdot}_\rho$ denotes 
the fixed $G$-invariant Hermitian inner-product on $V_\rho$.
Since the vectors
\[
\frac{1}{|Y|} \sum_{y \in Y} \rho(y^{-1})v_1,\quad \frac{1}{|Y|} \sum_{y \in Y} \rho(y)v_2
\]
are both $K$-invariant in $V_\rho$ and $V_\rho^K$ is of one-dimensional, 
one can find $\lambda_1,\lambda_2 \in \mathbb{C}$ such that 
\begin{align*}
\lambda_1 w &= \frac{1}{|Y|} \sum_{y \in Y} \rho(y^{-1})v_1, \\
\lambda_2 w &= \frac{1}{|Y|} \sum_{y \in Y} \rho(y)v_2.
\end{align*}
Then we have 
\begin{align*}
\lrangle*{v_1, \frac{1}{|X|} \sum_{x \in X} \rho(x) v_2}_\rho 
  &= \frac{1}{|X|} \lrangle*{v_1,  \sum_{z \in \Omega} \sum_{y_1,y_2 \in Y} \rho(y_1) \rho(z) \rho(y_2) v_2}_\rho \\
  &= \frac{1}{|\Omega| |Y|^2} \sum_{z \in \Omega} \lrangle*{\sum_{y_1 \in Y} \rho(y_1^{-1}) v_1,  \rho(z)\sum_{y_2 \in Y} \rho(y_2) v_2}_\rho \\
  &= \lambda_1 \overline{\lambda_2}  \frac{1}{|\Omega|} \sum_{z \in \Omega} \lrangle*{w,\rho(z)w}_\rho \\
  &= \lambda_1 \overline{\lambda_2} \frac{1}{|\Omega|} \sum_{z \in \Omega} Z^{(G,K)}_\rho(z) \\
  &= 0.
\end{align*}
This completes the proof.
\end{proof}

In order to state our theorem more general situation, 
we use the following notation:

\begin{definition}
For a fixed subset $\Upsilon$ of $\widehat{G}^*_K$,
we say that 
a non-empty finite multi-set $\Omega$ on $G$ has the property $\Upsilon$
if 
\[
\sum_{z \in \Omega} Z_\rho^{(G,K)}(z) = 0 \text{ for any } \rho \in \Upsilon.
\]
\end{definition}

Let us state our inductive constructions of designs on $G$ below: 

\begin{corollary}\label{corollary:const_gen}
We fix a collection $\Lambda$ of finite-dimensional $G$-representations
and a $(\Lambda|_K)$-design $Y$ on $K$,
where we put 
\[
\Lambda|_K := \{ \tau|_K \mid \tau \in \Lambda \}.
\]
\begin{enumerate}
\item If $(\supp \Lambda) \cap \widehat{G}_K^* = \emptyset$ (see \cref{subsection:definition_designs} for the notation of $\supp \Lambda$), 
then $Y$ itself is a $\Lambda$-design on $G$.
\item Let us consider the case where $(\supp \Lambda) \cap \widehat{G}_K^* \neq \emptyset$.
Take a finite decomposition 
\[
(\supp \Lambda) \cap \widehat{G}^*_K = \bigsqcup_{i=1}^l \Upsilon_i
\]
of $(\supp \Lambda) \cap \widehat{G}_K^*$. 
For each $i = 1,\dots,l$, 
we fix a non-empty finite multi-set $\Omega_i$ on $G$ 
with the property $\Upsilon_i$.
Then a non-empty finite multi-set 
\[
X := Y \prod_{i=1}^l (\Omega_i \cdot Y)
\]
is a $\Lambda$-design on $G$.
\end{enumerate}
\end{corollary}

\cref{corollary:const_gen} above follows from \cref{theorem:gen_for_irred} and the lemma below:

\begin{lemma}\label{lemma:prod}
Let $\{ \Upsilon_i \}_{i=1,\dots,l}$ be a finite family of 
collections of finite-dimensional $G$-representations, 
$Y$ a non-empty finite multi-set on $G$
and $\{ \Omega_i \}_{i = 1,\dots,l}$ a family of non-empty finite multi-sets on $G$.
Suppose that $Y \cdot \Omega_i \cdot Y$ is a $\Upsilon_i$-design on $G$
for each $i = 1,\dots,l$.
Then 
\[
Y \cdot \prod_{i=1}^l (\Omega_i \cdot Y)
\]
is a $(\bigcup_{i=1}^l \Upsilon_i)$-design on $G$.
\end{lemma}

\cref{lemma:prod} can be proved easily by applying \cref{proposition:designs_on_products}.

\subsection{Some remarks to apply our methods}

In the setting of \cref{corollary:const_gen}, 
by \cref{proposition:zero_zonal}, 
we can choose a decomposition 
\[
(\supp \Lambda) \cap \widehat{G}^*_K = \bigsqcup_{i=1}^l \Upsilon_i.
\]
satisfying that 
one of the following two conditions holds 
for each $i = 1,\dots,l$:
\begin{description}
\item[Condition (1)] $\{ Z_\rho^{(G,K)} \}_{\rho \in \Upsilon_i}$ has common zeroes in $G$.
\item[Condition (2)] $\{ \Re Z_\rho^{(G,K)} \}_{\rho \in \Upsilon_i}$ has common zeroes in $G$ but $\{ Z_\rho^{(G,K)} \}_{\rho \in \Upsilon_i}$ has no common zeroes in $G$. 
\end{description}
In such situations, as a corollary to \cref{proposition:inv_conj},
for each $i = 1,\dots,l$,
we can take $\Omega_i$ with the property $\Upsilon_i$ as below 
\begin{itemize}
\item If $\Upsilon_i$ satisfies Condition (1) above, 
then $\Omega_i$ can be taken as $\{ \omega \}_{\mathrm{mult}}$ for 
a common zero point $\omega \in G$ for $\{ Z_\rho^{(G,K)} \}_{\rho \in \Upsilon_i}$.
\item If $\Upsilon_i$ satisfies Condition (2) above, 
then $\Omega_i$ can be taken as $\{ \omega,\omega^{-1} \}_{\mathrm{mult}}$ for 
a common zero point $\omega \in G$ for $\{ \Re Z_\rho^{(G,K)} \}_{\rho \in \Upsilon_i}$.
Note that in this situation, $\omega \neq \omega^{-1}$ because of \cref{proposition:inv_conj}.
\end{itemize}

\begin{remark}
Basically, the cardinality of our design $X$ in \cref{corollary:const_gen} is terribly huge if $l$ is large, see \cref{rmk:radiative}.
Therefore, we want to find a  decomposition 
\[
(\supp \Lambda) \cap \widehat{G}^*_K = \bigsqcup_{i=1}^l \Upsilon_i.
\]
with small $l$
and mult-set $\Omega_i$ with the property $\Upsilon_i$ for each $i$.
However, for general subset $\Upsilon$ of $\widehat{G}_K^*$, 
the construction problem of $\Omega$ with the property $\Upsilon$ is not easy.
\end{remark}

\section{Reduction to classic design theory}

The object we investigated in this paper are $\Lambda$-designs on compact group $G$. 
Classic design theory focuses on $\mathcal{F}$-designs on a space $M$, where $\mathcal{F}$ is a functional space. 
We will reduce our designs to classic designs through matrix coefficients of the $G$-representations. 

Let $G$ be a compact Hausdorff group equipped with the probability Haar measure as in \cref{subsection:Haar_Bochner}.
In \cref{section:designs}, we give a definition of designs on $G$
in terms of averages of operators on representations.
In this section, we give a characterization of designs on $G$ in terms of averages of functions on $G$.

The space of all continuous $\CC$-valued functions on $G$ 
is denoted by $C(G)$.

For each $(\rho,V_\rho) \in \widehat{G}$, 
we define a $\CC$-linear map 
\[
\Phi_\rho : V_\rho \otimes V_\rho^{\vee} \rightarrow C(G)
\]
by 
\[
\Phi_\rho(v \otimes \eta) : G \rightarrow \CC, ~ g \mapsto \langle \rho(g^{-1})v,\eta \rangle
\]
The function $\Phi_\rho(v \otimes \eta)$ is called the matrix coefficient of $\rho$ at $(v,\eta)$.

It is well known that $\Phi_\rho$ is injective 
and let us denote the image $\Phi_\rho(V_\rho \otimes V_\rho^{\vee})$ by $C(G)_{\rho \boxtimes \rho^\vee}$.

\begin{remark}[{\cite[Theorem 5.12]{Folland95AbstractHA}}] \label{rmk:Heterotrichales}
$C(G)$ is an infinite-dimensional $(G \times G)$-representation.
For each $\rho \in \widehat{G}$, 
$C(G)_{\rho \boxtimes \rho^\vee}$ is $(G \times G)$-stable finite-dimensional subspace of $C(G)$
which is equivalent to the irreducible $(G \times G)$-representation 
$\rho \boxtimes \rho^\vee$.
Peter--Weyl's theorem claims that 
\begin{itemize}
\item The family of subspaces 
$\{ C(G)_{\rho \boxtimes \rho^\vee} \}_{\rho \in \widehat{G}}$
of $C(G)$
is orthogonal with respect to the $L^2$-inner product induced from the Haar measure.
In particular, 
$\{ C(G)_{\rho \boxtimes \rho^\vee} \}_{\rho \in \widehat{G}}$
is linearly independent in $C(G)$.
\item The algebraic direct sum 
\[
\bigoplus_{\rho \in \widehat{G}} C(G)_{\rho \boxtimes \rho^\vee}
\]
is dense in $C(G)$ with respect to the supremum norm.
\end{itemize}
\end{remark}

For a collection $\Lambda$ of finite-dimensional $G$-representations.
We define the subspace $C(G)_{\Lambda}$ of $C(G)$ by 
\[
C(G)_{\Lambda} := \bigoplus_{\rho \in \supp \Lambda} C(G)_{\rho \boxtimes \rho^\vee}
\]
(see \cref{subsection:definition_designs} for the notation of $\supp \Lambda$). 
Then we have a characterization of $\Lambda$-designs on $G$ as follows:

\begin{proposition}
Let $\Lambda$ be a collection of finite-dimensional $G$-representations.
Then the following three conditions on non-empty finite multi-set $X$ on $G$ are equivalent:
\begin{description}
\item[Condition (i)] $X$ is a $\Lambda$-design on $G$.
\item[Condition (ii)] 
The equality below holds for any $f \in C(G)_{\Lambda}$:
\[
\frac{1}{|X|} \sum_{x \in X} f(x) = \int_{G} f(g) \dd g.
\]
\end{description}
\end{proposition}

\begin{proof} 
    By Peter-Weyl theorem, see \cref{rmk:Heterotrichales}, we may assume that $\abs{\supp \Lambda} = 1$. 

    Suppose $X$ is a $\rho$-design on $G$. 
    Then $X^{-1}$ is also a $\rho$-design on $G$.  
    Hence
    \[
        \frac{1}{|X|} \sum_{x \in X} \rho(x^{-1}) = \int_{G} \rho(g^{-1}) \dd g.
    \]
    We apply both sides to $\lrangle{- v, \eta}$ for $(v \otimes \eta) \in V_\rho \otimes V_\rho^{\vee}$. 
    \[
        \frac{1}{|X|} \sum_{x \in X} \lrangle{\rho(x^{-1})v,\eta} = \int_{G} \lrangle{\rho(g^{-1})v,\eta} \dd g.
    \]
    In other words 
    \[
        \frac{1}{|X|} \sum_{x \in X} f(x) = \int_{G} f(g) \dd g.
    \]
    for every $f \in C(G)_{\rho \boxtimes \rho^\vee}$. 

    Note that for $A,B \in \End_\CC(V)$, $A = B$ if and only if $\lrangle{Av, \eta} = \lrangle{Bv,\eta}$ for every $v \in V$ and $\eta \in V^\vee$.
    Therefore we can prove the other direction by reversing the above argument.
\end{proof}

\section{Different constructions of unitary 4-designs on U(4)} \label{sec:Schistosoma}

\subsection{Unitary group}

In this subsection we prepare the irreducible representations, Gelfand pair and zonal spherical functions which are used to construct unitary designs.

We call $X$ a strong unitary $t$-design if $X$ is a $\blacksquare_{n}^{t,t}$-design on $U(n)$. 

Let $G = U(n)$ and $K = U(m) \times U(n-m)$, then $(G,K)$ is a Gelfand pair.
The Gelfand pair $(G,K) = (U(n), U(m) \times U(n-m))$ has been studied carefully. 
Recall that \cref{thm:financiery} says that the irreducible representations of $U(n)$ are characterized by dominant weights $\lambda$. 
The $U(m) \times U(n-m)$-spherical representations of $U(n)$ are indexed by some `symmetric' dominant weights. 

\begin{theorem}[{\cite[Section 12.3.2, pp. 577-578]{MR2522486}}] \label{thm:stereometer}
  Let $\rho_\lambda$ be a spherical representation of $U(n)$ with respect to $U(m) \times U(n-m)$ where $m \leq n/2$, then $\lambda$ has the form 
  \begin{equation*}
    \lambda = (\lambda_1, \ldots, \lambda_m, 0, \ldots, 0, -\lambda_m, \ldots, -\lambda_1).
  \end{equation*}
  In other words the spherical representations are characterized by integer partitions into at most $m$ parts. 
  We denote by $\tilde{\lambda} = (\lambda_1, \ldots, \lambda_m)$.
\end{theorem}

For the construction of strong unitary design, we take $\Lambda = \blacksquare_n^{t,t}$. 

The zonal spherical functions $Z_{\rho}^{(G,K)}$ with $G = U(n)$ and $K = U(m) \times U(n-m)$ are given in \cref{sec:bitterishness}. 
Some common zeroes of the zonal spherical functions (with $K = U(m) \times U(n-m)$ or $K = \mathcal{X}_2$ the complex Clifford group of genus $2$) have been computed. 
See \cref{sec:gnathotheca} for detail.

Before giving the constructions, we note that there are several techniques to shrink the size of the unitary design.

\begin{observation} \label{obsv:schoolmistressy}
  Let $X$ be a $\Lambda$-design of size $\abs{X}$. 
  Suppose the multiplicities of each element in $X$ is a multiple of $D$. 
  Then we get a $\Lambda$-design of size $\abs{X}/D$.
\end{observation}

\begin{observation} \label{obsv:oystergreen}
  Let $X$ be a unitary $t$-design on $U(n)$ of size $\abs{X}$. 
  Suppose $X$ has a partition $X = \bigsqcup_{i=1}^k \omega_i Y$ where $\abs{\omega_i} = 1$ for every $i$. 
  Then $Y$ is a unitary $t$-design on $U(n)$ of size $\abs{X}/k$.
\end{observation}

\begin{proof}
  It follows from $(\omega U)^{\otimes t} \otimes ((\omega U)^{\dagger})^{\otimes t} = U^{\otimes t} \otimes (U^\dagger)^{\otimes t}$.
\end{proof}

\subsection{Inductive construction}

\begin{example} \label{exmp:undonnish}
  $X_1 = \set{1, \omega, \omega^2, \omega^3, \omega^4}$, where $\omega$ is a primitive $5$-th root of unity, is a strong unitary $4$-design on $U(1)$ of size $5$. 
\end{example}

\begin{example} \label{exmp:ancistrocladaceous}
  Let $X_{1,1} = \set*{\begin{bmatrix}
    g & 0 \\
    0 & h
  \end{bmatrix} : g,h \in X_1}$ and its size is $25$. 
  Note that 
  \begin{align*}
    \blacksquare_{2}^{4,4} \cap \widehat{U(2)}^*_{U(1) \times U(1)}&= \set{(1,-1),(2,-2),(3,-3),(4,-4)} \\
    &= \set{(1,-1),(3,-3)} \sqcup \set{(2,-2)} \sqcup \set{(4,-4)}.
  \end{align*}
  By \cref{corollary:const_gen,obsv:schoolmistressy,exmp:nephroerysipelas}, we have a strong unitary $4$-design $X_2$ on $U(2)$ of size $25^4/5^3 = 5^{5}$. 
  Then we get a unitary $4$-design on $U(2)$ of size $5^4$ by \cref{obsv:oystergreen}.
\end{example}

\begin{example} \label{exmp:toperdom}
  Let $X_{2,2} = \set*{\begin{bmatrix}
    g & 0 \\
    0 & h
  \end{bmatrix} : g, h \in X_2}$ and its size is $5^{5} \times 5^{5} = 5^{10}$. 
  Note that 
  \begin{align*}
    \blacksquare_{4}^{4,4} \cap \widehat{U(4)}^*_{U(2) \times U(2)}
    =& \set{(1,0,0,-1),(3,0,0,-3),(2,1,-1,-2),(3,1,-1,-3)} \\
    &\bigsqcup \set{(2,0,0,-2),(1,1,-1,-1)}\\
    &\bigsqcup \set{(4,0,0,-4),(2,2,-2,-2)}.
  \end{align*}
  By \cref{corollary:const_gen,obsv:schoolmistressy,exmp:toreador,exmp:kelebe,exmp:chylifactive}, we have a strong unitary $4$-design $X_4$ on $U(4)$ of size $(5^{10})^4/5^3 = 5^{37}$. 
  Then we get a unitary $4$-design on $U(4)$ of size $5^{36}$ by \cref{obsv:oystergreen}.
\end{example}

\subsection{Construction by finite group} \label{sec:pectineus}

We will display several known construction of unitary design. 
Then we will construct a unitary $4$-design by complex Clifford group. 
The character formulas are useful to verify whether a group is a unitary design.

\begin{theorem} \label{thm:bedragglement}
  Let $\chi_\mu$ be the character of irreducible representation $\rho_\mu$ of $U(n)$. Let $\lambda_1, \lambda_2, \ldots, \lambda_n$ be the eigenvalues of an element $g \in U(n)$. 
  The character is given by Schur polynomial. 
  \begin{equation}
    \chi_\mu(g) = S_{\mu}(\lambda_1, \ldots, \lambda_n).
  \end{equation}
\end{theorem}

\begin{remark}
  The Schur polynomial is a symmetric polynomial. 
  One can use Newton-Girard formulae to transform it into a polynomial in terms of power sums of $\lambda_1, \ldots, \lambda_n$. 
  Since $\sum_{i=1}^n \lambda_i^k = \Tr (g^k)$, we can avoid the computation of eigenvalues.
\end{remark}

\begin{example} \label{exmp:spiricle}
  By computation we have $\lrangle{\chi_\mu |_{SL(2,5)} , 1}_{SL(2,5)} = \lrangle{\chi_\mu , 1}_{U(2)}$ for every $\mu \in \blacksquare_{2}^{5,5}$. 
  Therefore $SL(2,5)$ is a strong unitary $5$-design on $U(2)$ of size $120$. 
  Then we get a unitary $5$-design on $U(2)$ of size $60$ by \cref{obsv:oystergreen}.
\end{example}

\begin{example} \label{exmp:reacknowledge}
  Let $X_{2,2} = \set*{\begin{bmatrix}
    g & 0 \\
    0 & h
  \end{bmatrix} : g, h \in SL(2,5)}$ and its size is $120 \times 120 = 120^2$. 
  By the same construction in \cref{exmp:toperdom}, we have a strong unitary $4$-design $X_4$ on $U(4)$ of size $(120^{2})^4/2^3 = 2^{21} 3^{8} 5^{8}$. 
  Then we get a unitary $4$-design on $U(4)$ of size $2^{20} 3^{8} 5^{8}$ by \cref{obsv:oystergreen}.
\end{example}

\begin{example} \label{exmp:nemocerous}
  \cite[Example 17]{Bannai_2019} gives a construction of unitary $3$-design on $U(3)$ by $SL(3,2)$ which is of size $168^2/4 = 7,056$. 
\end{example}

\begin{example} \label{exmp:resorcinum}
  \cite[Example 18]{Bannai_2019} gives a construction of unitary $4$-design on $U(4)$ by $Sp(4,3)$ which is of size $25,920^2/6 = 447,897,600$. 
\end{example}

The complex Clifford group $\mathcal{X}_2 \leq U(4)$ is a unitary $3$-design \cite{Webb:2016:CGF:3179439.3179447} and fails gracefully to be a unitary $4$-design \cite{1609.08172}. 
Using \cref{thm:bedragglement}, we can compute
\begin{equation}
  \lrangle{\chi_\mu |_{\mathcal{X}_2} , 1} =
  \begin{cases}
    0, & \mu \in \blacksquare_{4}^{4,4} \backslash \set{(4,0,0,-4), (2,2,-2,-2), (0,0,0,0)}; \\
    1, & \mu = (4,0,0,-4) \text{ or } (2,2,-2,-2) \text{ or } (0,0,0,0).
  \end{cases}
\end{equation}
This implies that $\mathcal{X}_2$ is a $\rho_\mu$-design for every $\mu \in \blacksquare_4^{4,4}$ except $\mu = (4,0,0,-4)$ and $\mu = (2,2,-2,-2)$. 
And the trivial representation of $\mathcal{X}_2$ is of multiplicity $1$ in the decomposition of $\rho_{(4,0,0,-4)}$ and $\rho_{(2,2,-2,-2)}$ into irreducible representations of $\mathcal{X}_2$. 
By \cref{corollary:const_gen}, $\mathcal{X}_2 g_{(4,0,0,-4)} \mathcal{X}_2 g_{(2,2,-2,-2)} \mathcal{X}_2$ is a strong unitary $4$-design where $g_{\lambda}$ is a zero of the $\mathcal{X}_2$-biinvariant function in $\rho_\lambda$. 

\begin{example} \label{exmp:photoceramist}
  $\mathcal{X}_2 g_{(4,0,0,-4)} \mathcal{X}_2 g_{(2,2,-2,-2)} \mathcal{X}_2$ is a strong unitary $4$-design on $U(4)$ of size $\abs{\mathcal{X}_2}^3$.
\end{example}

\begin{example}
  \cref{exmp:photoceramist} might be improved by \cref{obsv:schoolmistressy,obsv:oystergreen}. 
  If the $\mathcal{X}_2$-biinvariant functions in $\rho_{(4,0,0,-4)}$ and $\rho_{(2,2,-2,-2)}$ has a common zero $g_c$, then $\mathcal{X}_2 g_c \mathcal{X}_2$ is a strong unitary $4$-design (unitary $4$-design) on $U(4)$ of size $92,160^2 / 8 = 1,061,683,200$ (of size $92,160^2 / 8^2 = 132,710,400$).
\end{example}

\subsection{Comparison between different unitary designs}

We denote by $L(n,t)$ the minimum size of a strong unitary $t$-design on $U(n)$ and by $l(n,t)$ the minimum size of a unitary $t$-design on $U(n)$. 
By \cref{corollary:const_gen} we have the following theorem.

\begin{theorem} \label{thm:tectocephalic}
  $l(n,t) \leq L(n,t) \leq (L(m,t)L(n-m,t))^{\abs{\widetilde{\Lambda}(m,t) }+1}$ for $1 \leq m \leq n/2$.
\end{theorem}

\begin{corollary}

  \begin{align*}
    L(1,t) & = t+1 \\
    L(2,t) &\leq (t+1)^{2(t+1)} \\
    L(3,t) &\leq (t+1)^{(2(t+1)+1)(t+1)} \\
    L(4,t) &\leq (t+1)^{((2(t+1)+1)(t+1)+1)(t+1)} \\
    L(4,t) &\leq (t+1)^{4(t+1)(\abs{\widetilde{\Lambda}(2,t)}+1)}
  \end{align*}

In particular we have $L(2,4) \leq 5^{10}$, $L(3,3) \leq 4^{36}$, $L(3,4) \leq 5^{55}$, $L(4,4) \leq 5^{280}$ and $L(4,4) \leq 5^{180}$. 
\end{corollary}

\begin{remark} \label{rmk:radiative}
  The size of $\widetilde{\Lambda}(m,t)$ can be estimated by partition function. 
  Let $p(k)$ be the number of possible partitions of a non-negative integer $k$. 
  It is known that 
  \begin{equation}
    p(k) \sim \frac{1}{4k \sqrt{3}} \exp \lrpar*{\pi \sqrt{\frac{2k}{3}}} \text{ as } k \to \infty
  \end{equation}
  This gives us $\abs{\widetilde{\Lambda}(m,t)} \leq \sum_{k=1}^t p(k)$. 
  The equality holds for $t \leq m$.
  On the other hand, if we fix $m$, then $$\abs{\widetilde{\Lambda}(m,t)} \leq \sum_{k=1}^t \binom{m+k-1}{m-1} = \binom{m+t}{m} - 1 = \mathcal{O}(t^m).$$
\end{remark}

\begin{example} \label{exmp:sortation}
  \cref{exmp:undonnish,exmp:ancistrocladaceous,exmp:toperdom} gives us the following upper bounds.
  \begin{align*}
    L(1,4) &= 5 & l(1,4) &= 1\\
    L(2,4) &\leq 25^4 / 5^3 = 5^5 & l(2,4) &\leq 5^5 / 5 = 5^4\\ 
        L(4,4) &\leq (5^5 \times 5^5)^4 / 5^3 = 5^{37} & l(4,4) &\leq 5^{37} / 5 = 5^{36}.
  \end{align*}
\end{example}

\begin{example}
  \cref{exmp:spiricle,exmp:reacknowledge} gives us the following upper bounds.
  \begin{align*}
    L(2,4) &\leq L(2,5) \leq 120 & l(2,4) &\leq l(2,5) \leq 120/2 = 60 \\
        L(4,4) &\leq (120^2)^4 / 2^3 = 2^{21} 3^{8} 5^{8} & l(4,4) &\leq (120^2)^4 / 2^4 = 2^{20} 3^{8} 5^{8}
  \end{align*}
\end{example}

\begin{example}
  \cref{exmp:nemocerous,exmp:resorcinum} shows that $l(3,3) \leq 7,056$ and $l(4,4) \leq 447,897,600$.
\end{example}

We have $132,710,400 < 447,897,600 < 2^{20} 3^{8} 5^{8} < 5^{36}$.
Therefore the construction by complex Clifford group $\mathcal{X}_2$ is the smallest if the common zero exists.

\section{Design on orthogonal group and sphere} \label{sec:Liassic}

The method of inductive construction applies to orthogonal group as well.

To apply \cref{corollary:const_gen} we need the following theorem on irreducible representations of the orthogonal group. 

\begin{theorem} \label{thm:exacerbescent}
  \begin{enumerate}
    \item \cite[Theorem 5.7A, 5.7C]{MR0000255} The irreducible representations of $O(n)$ are indexed by admissible Young diagram $\lambda$ such that $f_1 + f_2 \leq n$ where $f_i$ is the number of squares in the $i$-th column of $\lambda$.
    Two admissible Young diagrams $\lambda$ and $\lambda'$ are called associated if $f_1 + f_1' = n$ and $f_i = f_i'$ for $i \geq 2$. 
    In particular, if $\lambda$ is associated to itself, then we call $\lambda$ self-associated. 
    \item \cite[Theorem 5.9A]{MR0000255} Under the restriction to special orthogonal group $SO(n)$, each irreducible representation $\rho_\lambda$ of $O(n)$ stays irreducible unless $\lambda$ is self-associated, in which case it decomposes into two irreducible representation of equal degree. 
    Associated representations become equivalent but no other equivalence are created. 
    \item \cite[Theorem 5.9A]{MR0000255} The irreducible representations of $SO(2p+1)$ are indexed by integer sequences $m : m_1 \geq m_2 \geq \cdots \geq m_p \geq 0$ and the irreducible representations of $SO(2p)$ are indexed by integer sequences $m : m_1 \geq m_2 \geq \cdots \geq m_{p-1} \geq \abs{m_p}$. 
    Here $m_i$ corresponds to the number of squares in the $i$-th row of admissible Young digram $\lambda$. 
    The negative $m_p$ accounts for the branching of self-associated Young digram. 
    \item \cite[Theorem 12.1b]{MR0148766} Let $m : m_1 \geq m_2 \geq \cdots \geq m_p \geq 0$ and $m' : m_1' \geq m_2' \geq \cdots \geq m_{p-1}' \geq \abs{m_p'}$ corresponds to irreducible representations of $SO(2p+1)$ and $SO(2p)$ respectively. 
    We have the following decomposition. 
    \begin{equation*}
      \rho_m = \bigoplus_{m'} \rho_{m'}
    \end{equation*}
    where the summation goes through $m'$ such that $m_1 \geq m_1' \geq m_2 \geq m_2' \geq \cdots \geq m_p \geq \abs{m_p'}$. 
    \item \cite[Theorem 12.1a]{MR0148766} Let $m : m_1 \geq m_2 \geq \cdots \geq \abs{m_p}$ and $m' : m_1' \geq m_2' \geq \cdots \geq m_{p-1}' \geq 0$ corresponds to irreducible representations of $SO(2p)$ and $SO(2p-1)$ respectively. 
    We have the following decomposition. 
    \begin{equation*}
      \rho_m = \bigoplus_{m'} \rho_{m'}
    \end{equation*}
    where the summation goes through $m'$ such that $m_1 \geq m_1' \geq m_2 \geq m_2' \geq \cdots \geq m_{p-1} \geq \abs{m_p}$.
    \item \cite[Theorem 3.2]{MR1941981} \cite[Section 12.3.2, pp. 575-576]{MR2522486} Let $\rho_\lambda$ be a spherical representation of $O(n)$ with respect to $O(m) \times O(n-m)$ where $m \leq n/2$, then $\lambda$ has the form 
    \begin{equation*}
      \lambda = (\lambda_1, \ldots, \lambda_m, 0, \ldots, 0)
    \end{equation*}
    where $\lambda_i \equiv 0 \pmod 2$ for every $i$.
    In other words the spherical representations are characterized by even integer partitions into at most $m$ parts. 
  \end{enumerate}
\end{theorem}

We can take $\Lambda(n,t) = \set{\lambda : \abs{\lambda} \leq t, f_1 + f_2 \leq n}$ and apply \cref{corollary:const_gen} to construct orthogonal $t$-design. 
Zonal spherical functions for $O(n)/O(m) \times O(n-m)$ are given in \cite[Theorem 15.1]{MR0374523}.
Every orbit of an orthogonal $t$-design is a spherical $t$-design because $\set{(k,0, \ldots, 0) : 0 \leq k \leq t} \subset \Lambda(n,t) $.

\section{Discussion} \label{sec:swainsona}

If we take $m=1$ in the inductive construction, then the zonal spherical function with respect to the complex Grassmannian $\mathcal{G}_{1,n}$ is a polynomial in one variable $y$. 
Therefore the root is an algebraic number. 
But this would result in very large designs as suggested by \cref{thm:tectocephalic}. 
If we take $m \approx n/2$, the size of the design could be smaller. 
Though we may not be able to express the unitary matrix by algebraic numbers, we can use bisection method to obtain as much precision as we want. 
The existence of exact unitary design close to the numerical approximation is guaranteed, which is not true for approximate unitary design. 
Our inductive construction also gives an upper bound of the size of unitary design. 
It would be interesting to study the asymptotic lower bound on the size of unitary designs just as the case of spherical designs.

\section*{Acknowledgement}

The authors are grateful to Shingo Kukita, Mikio Nakahara, Yan Zhu for useful discussion and comments. 
We thank the following universities for providing us the place of discussions on this and related topics: Shanghai Jiao Tong University in particular Yaokun Wu, Shanghai University in particular Mikio Nakahara, and TGMRC in China Three Gorges University in Yichang in particular Zongzhu Lin. 
YN was supported by JST, PRESTO Grant Number JPMJPR1865.

\bibliographystyle{plain}

\bibliography{construction}

\appendix

\section{zonal polynomials on complex Grassmannian} \label{sec:bitterishness}

\subsection{general formula of zonal polynomials on complex Grassmannian}

The cosets of $U(m) \times U(n-m)$ in the unitary group $U(n)$ are naturally identified with the complex Grassmannian $\mathcal{G}_{m,n}$, the collection of $m$-dimensional subspaces of the $n$-dimensional linear space. 
Let $W \in \mathcal{G}_{m,n}$ be an $m$-dimensional subspace of an $n$-dimensional space $V$. 
By fixing a basis of $V$, we may use the projection matrix $P_W$ to represent the space $W$. 
The unitary group $U(n)$ acts on the complex Grassmannian $\mathcal{G}_{m,n}$ by $U : P_W \mapsto U P_W U^\dagger$. 
It induces an action of $U(n)$ on $\mathcal{G}_{m,n} \times \mathcal{G}_{m,n}$. 
The orbits of pairs of subspaces are characterized by the `angles' between the subspaces. 

\begin{definition}
  Let $X,Y \in \mathcal{G}_{m,n}$ be two subspaces of $V$. 
  The principal angles $\theta_1, \ldots, \theta_m$ between $X$ and $Y$ are defined as follows. 
  Let $x_1$ and $y_1$ be unit vectors in $X$ and $Y$ respectively. 
  Then $\theta_1$ is the smallest achievable angle between $x_1$ and $y_1$. 
  Formally 
  \begin{equation*}
    \theta_1 = \min \set{\arccos{\abs{\lrangle{x_1,y_1}}} : x_1 \in X, y_1 \in Y, \norm{x_1} = \norm{y_1} = 1}.
  \end{equation*} 
  Let $x_2$ and $y_2$ be unit vectors in $X \cap \Span\set{x_1}^\perp$ and $Y \cap \Span\set{y_1}^\perp$ respectively. 
  Then $\theta_2$ is the smallest achievable angle between $x_2$ and $y_2$. 
  Similarly we define $\theta_3$, $\theta_4$, up to $\theta_m$.
\end{definition}

The first $m$ eigenvalues of the matrix $P_X P_Y$ are exactly $\cos^2 \theta_1, \ldots, \cos^2 \theta_m$, and the remaining $n-m$ eigenvalues are $0$. 
These principal angles not only characterize the orbitals \cite[Lemma 2.1]{MR2577476}, but also determine the value of zonal functions. 

For the ease of notation, we use $y_i = \cos^2 \theta_i$ when referring to the principal angles. We also use $\kappa$ and $\sigma$ instead of $\widetilde{\lambda}$ for integer partitions into at most $m$ parts. 
Let us introduce Schur polynomial and some coefficients before giving the formula of zonal spherical functions. 

\begin{definition}
  Let $y = (y_1, y_2, \ldots, y_m)$ be the variables and $\sigma = (s_1, s_2, \ldots, s_m)$ be an non-increasing integer sequence of length $m$. 
  The Schur polynomial is a symmetric polynomial given by 
  \begin{equation}
    S_\sigma(y) = \frac{\det (y_i^{s_j+m-j})_{i,j}}{\det (y_i^{k-j})_{i,j}}.
  \end{equation}
  And the normalized Schur polynomial $S_{\sigma}^*(y)$ is a scaling of $S_{\sigma}(y)$ such that $S_{\sigma}^*(1, \ldots, 1) = 1$.
\end{definition}

\begin{definition}[{\cite[Theorem 5.3]{MR2577476}}]
  The ascending factorial is defined by
  \begin{equation}
    (a)^{\bar{s}} = a (a+1) \cdots (a+s-1).
  \end{equation}
  The complex hypergeometric coefficients are defined by 
  \begin{equation}
    [a]^{\sigma} = \prod_{i=1}^n (a-i+1)^{\bar{s}_i}
  \end{equation}
  where $\sigma = (s_1, s_2, \ldots, s_m)$ is a partition.

  The complex hypergeometric binomial coefficients $\hyperBinom{\kappa}{\sigma}$ are given by the expansion 
  \begin{equation}
    S_\kappa^*(y+1) = \sum_{\sigma \leq \kappa} \hyperBinom{\kappa}{\sigma} S_\sigma^*(y)
  \end{equation}
  where $y+1 = (y_1 + 1, \ldots, y_m + 1)$ and $(s_1, s_2, \ldots, s_m) = \sigma \leq \kappa = (k_1, k_2, \ldots, k_m)$ if and only if $s_i \leq k_i$ for every $i$. 
\end{definition}

\begin{theorem}[{\cite[Theorem 16.1]{MR0374523} \cite[Theorem 5.3]{MR2577476}}]
  Let $\kappa = \widetilde{\lambda}$ be the index of a spherical representation of $U(n)$ with respect to $U(m) \times U(n-m)$. 
  Let $\sigma = (s_1, \ldots, s_m)$ and $\kappa = (k_1, \ldots, k_m)$ be partitions of $s$ and $k$ respectively. 
  Let $\sigma + e_i = (s_1, \ldots, s_{i-1}, s_i +1, s_{i+1}, \ldots, s_m)$, and we call $\sigma + e_i$ valid if it is non-increasing. 
  We further put
  \begin{equation}
    \rho_\sigma = \sum_{i=1}^m s_i (s_i - 2i + 1),
  \end{equation}
  and 
  \begin{equation} \label{eqn:aphlaston}
    [c]_{(\kappa, \sigma)} = \sum_{i: \sigma + e_i \text{ is valid}} \frac{\hyperBinom{\kappa}{\sigma+e_i} \hyperBinom{\sigma + e_i}{\sigma}}{(k-s) \hyperBinom{\kappa}{\sigma}} \cdot \frac{[c]_{(\kappa, \sigma + e_i)}}{c + \frac{\rho_\kappa - \rho_\sigma}{k-s}}
  \end{equation}
  where the summation is over valid partitions $\sigma + e_i$. 

  Then up to a scaling, the zonal spherical function is given by
  \begin{equation}
    Z_{\kappa} (y) = \sum_{\sigma \leq \kappa} \frac{(-1)^s \hyperBinom{\kappa}{\sigma} [n]_{(\kappa, \sigma)}}{[m]_\sigma} S_{\sigma}^* (y)
  \end{equation}
\end{theorem}

\begin{remark}
  In \cite[Formula 16.6]{MR0374523} the expression for $[c]_{\kappa, \sigma}$ has an extra factor $\hyperBinom{\kappa}{\sigma}$, which is redundant. 
\end{remark}

\subsection{Examples of zonal spherical functions for complex Grassmannian} \label{apdx:chold}

For the notation in the following formulas, please see \cref{sec:pectineus}.
\begin{align*}
    Z_{0}(y) &= S_{0}^*(y) = 1 \\
    Z_{1}(y) &= S_{0}^*(y) - \frac{n}{m} S_{1}^*(y) \\
    Z_{2}(y) &= S_{0}^*(y) - \frac{n+1}{m} S_{1}^*(y) + \frac{(n+1)(n+1)}{m(m+1)} S_{2}^*(y) \\
    Z_{3}(y) &= S_{0}^*(y) - \frac{3 (n+2)}{m} S_{1}^*(y) + \frac{3(n+2)(n+3)}{m(m+1)} S_{2}^*(y) - \frac{(n+2)(n+3)(n+4)}{m(m+1)(m+2)} S_{3}^*(y) \\
    Z_{4}(y) &= S_{0}^*(y) - \frac{4 (n+3)}{m} S_{1}^*(y) + \frac{6(n+3)(n+4)}{m(m+1)} S_{2}^*(y) - \frac{4(n+3)(n+4)(n+5)}{m(m+1)(m+2)} S_{3}^*(y) \\
    & + \frac{(n+3)(n+4)(n+5)(n+6)}{m(m+1)(m+2)(m+3)} S_{4}^*(y) 
\end{align*}
\begin{align*}
    Z_{1,1}(y) &= S_{0}^*(y) - \frac{2 (n-1)}{m} S_{1}^*(y) + \frac{(n-2)(n-1)}{(m-1)m} S_{1,1}^*(y) \\
    Z_{2,1}(y) &= S_{0}^*(y) - \frac{3 n}{m} S_{1}^*(y) + \frac{3 (n-2)n}{2(m-1)m} S_{1,1}^*(y) + \frac{3 n(n+2)}{2m(m+1)} S_{2,0}^*(y) - \frac{(n-2)n(n+2)}{(m-1)m(m+1)} S_{2,1}^*(y) \\
    Z_{3,1}(y) &= S_{0}^*(y) - \frac{4 (n+1)}{m} S_{1}^*(y) + \frac{2 (n-2)(n+1)}{(m-1)m} S_{1,1}^*(y) + \frac{4 (n+1)(n+3)}{m(m+1)} S_{2,0}^*(y) \\ &- \frac{8(n-2)(n+1)(n+3)}{3(m-1)m(m+1)} S_{2,1}^*(y) - \frac{4(n+1)(n+3)(n+4)}{3m(m+1)(m+2)} S_{3,0}^*(y) \\ &+ \frac{(n-2)(n+1)(n+3)(n+4)}{(m-1)m(m+1)(m+2)} S_{3,1}^*(y) \\
    Z_{2,2}(y) &= S_{0}^*(y) - \frac{4 n}{m} S_{1}^*(y) + \frac{3 (n-1)}{(m-1)m} S_{1,1}^*(y) + \frac{3 n(n+1)}{m(m+1)} S_{2}^*(y) \\ 
    &- \frac{4 (n-1)n(n+1)}{(m-1)m(m+1)} S_{2,1}^*(y) + \frac{(n-1)n^2 (n+1)}{(m-1)m^2 (m+1)} S_{2,2}^*(y) 
\end{align*}
\begin{align*}
    Z_{1,1,1,1}(y) &= S_{0}^*(y) - \frac{4 (n-3)}{m} S_{1}^*(y) + \frac{6(n-4)(n-3)}{(m-1)m} S_{1,1}^*(y) - \frac{4(n-5)(n-4)(n-3)}{(m-2)(m-1)m} S_{1,1,1}^*(y) \\
    &+ \frac{(n-6)(n-5)(n-4)(n-3)}{(m-3)(m-2)(m-1)m} S_{1,1,1,1}^*(y)\\
\end{align*}

\section{Examples of zonal spherical functions with common zero} \label{sec:gnathotheca}

\cref{exmp:nephroerysipelas,exmp:toreador,exmp:kelebe,exmp:chylifactive} provide the common zeroes of certain zonal spherical functions on complex Grassmannian. 
\cref{exmp:herniotomy} study the potential common zero of two zonal spherical functions with respect to the complex Clifford group.

\begin{example} \label{exmp:nephroerysipelas}
  The zonal spherical functions $Z_{1}(y)$ and $Z_{3}(y)$ with $m=1$ and $n=2$ has a common zero $y_1 = 1/2$. 
\end{example}

\begin{example} \label{exmp:toreador}
  The common zeroes of $Z_{2}(y)$ and $Z_{1,1}(y)$ with $m=2$ and $n=4$ can be given as algebraic numbers. 
  The equation 
  \begin{equation} \label{eqn:irian}
    \begin{aligned}
      0 &= 90 t^4-180 t^3+114 t^2-24 t+1
    \end{aligned}
  \end{equation}
  has $4$ real roots. 
  We may order the roots from the smallest to the largest.
  Let $y_1$ and $y_2$ be the 2nd and 4th zeroes of \cref{eqn:irian}, then $$(y_1, y_2) = \left(\frac{1}{2} \left(1-\sqrt{\frac{1}{15} \left(7-2 \sqrt{6}\right)}\right), \frac{1}{2} \left(1+\sqrt{\frac{1}{15} \left(7+2 \sqrt{6}\right)}\right) \right)$$ is one common zero of $Z_{2}(y)$ and $Z_{1,1}(y)$. 
  The 4 common zeroes are shown in \cref{fig:theotechnic} as intersection of the loci of $Z_{2}(y)$ and $Z_{1,1}(y)$ for $\mathcal{G}_{2,4}$.
\end{example}

\begin{figure}[htbp]
\centering
\includegraphics[scale=0.7]{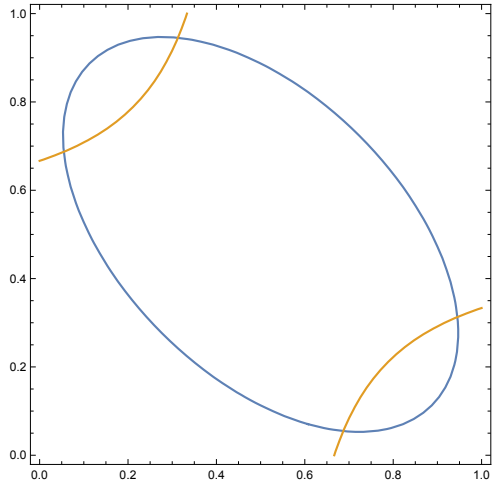}
\caption{Loci of $Z_{2}(y)$ and $Z_{1,1}(y)$ for $\mathcal{G}_{2,4}$}   \label{fig:theotechnic}
\end{figure}

\begin{example} \label{exmp:kelebe}
  The common zeroes of $Z_{1}(y)$, $Z_{3}(y)$ and $Z_{2,1}(y)$ with $m=2$ and $n=4$ form a line segment $y_1+y_2 = 1, 0 \leq y_1, y_2 \leq 1$. 
  Together with $Z_{3,1}(y)$, the four zonal spherical functions have common zeros.
  The equation 
  \begin{equation} \label{eqn:bozo}
    \begin{aligned}
      0 &= 70 t^4-140 t^3+90 t^2-20 t+1
    \end{aligned}
  \end{equation}
  has $4$ real roots. 
  We may order the roots from the smallest to the largest.
  Let $y_1$ and $y_2$ be the 2nd and 3rd zeroes of \cref{eqn:bozo}, then $$(y_1, y_2) = \left(\frac{1}{2} \left(1-\sqrt{\frac{1}{35} \left(15-2 \sqrt{30}\right)}\right), \frac{1}{2} \left(1+\sqrt{\frac{1}{35} \left(15-2 \sqrt{30}\right)}\right) \right)$$ is one common zero of $Z_{2}(y)$ and $Z_{1,1}(y)$. 
  The 4 common zeroes are shown in \cref{fig:Ppolyphloisboioism} as intersection of the loci of $Z_{3,1}(y)$ and $Z_{1}(y)$ for $\mathcal{G}_{2,4}$.
\end{example}

\begin{figure}[htbp]
\centering
\includegraphics[scale=0.7]{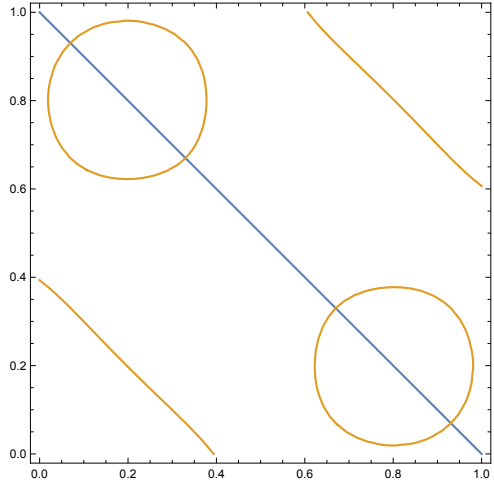}
\caption{Loci of $Z_{1}(y)$ and $Z_{3,1}(y)$ for $\mathcal{G}_{2,4}$}   \label{fig:Ppolyphloisboioism}
\end{figure}

\begin{example} \label{exmp:chylifactive}
  The common zeroes of $Z_{4}(y)$ and $Z_{2,2}(y)$ with $m=2$ and $n=4$ can be given as algebraic numbers. 
  The equation 
  \begin{equation} \label{eqn:looter}
    \begin{aligned}
      0 &=11430720000 t^{16}-91445760000 t^{15}+332951472000 t^{14}-730359504000 t^{13} \\
    &+1076946091200 t^{12}-1127785075200 t^{11}+863978226720 t^{10}-491476389600 t^9 \\
    &+208573299152 t^8-65783614208 t^7+15232863368 t^6-2533271096 t^5 \\
    &+292023188 t^4-22052192 t^3+993302 t^2-22634 t+197 
    \end{aligned}
  \end{equation}
  has $16$ real roots. 
  We may order the roots from the smallest to the largest.
  Let $y_1$ and $y_2$ be the 4th and 10th zeroes of \cref{eqn:looter}, then $(y_1, y_2) \approx (0.155944,0.648664)$ is one common zero of $Z_{4}(y)$ and $Z_{2,2}(y)$. 
  The 16 common zeroes are shown in \cref{fig:Propionibacterieae} as intersection of the loci of $Z_{1}(y)$ and $Z_{3,1}(y)$ for $\mathcal{G}_{2,4}$.
\end{example}

\begin{figure}[htbp]
\centering
\includegraphics[scale=0.7]{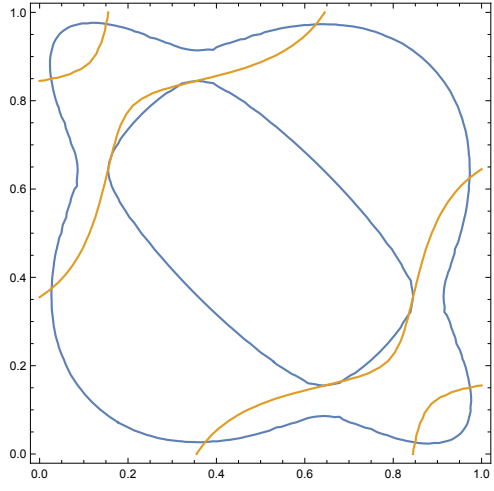}
\caption{Loci of $Z_{4}(y)$ and $Z_{2,2}(y)$ for $\mathcal{G}_{2,4}$}   \label{fig:Propionibacterieae}
\end{figure}

\begin{example} \label{exmp:herniotomy}
  Note that the $\mathcal{X}_2$-biinvariant function in $\rho_\lambda$ is equal to 
  \begin{equation}
    \frac{1}{\abs{\mathcal{X}_2}^2} \sum_{h_1,h_2 \in \mathcal{X}_2} \chi_\rho(h_1 g h_2) = \frac{1}{\abs{\mathcal{X}_2}} \sum_{h \in \mathcal{X}_2} \chi_\rho(g h).
  \end{equation}
  Therefore we can minimize 
  \begin{equation}
    I(g) = \frac{1}{2}\lrpar*{\abs*{\frac{1}{\abs{\mathcal{X}_2}} \sum_{h \in \mathcal{X}_2} \chi_{(4,0,0,-4)}(g h)} + \abs*{\frac{1}{\abs{\mathcal{X}_2}} \sum_{h \in \mathcal{X}_2} \chi_{(2,2,-2,-2)}(g h)}}
  \end{equation}
  by random optimization (See \cref{apdx:ahu} \cref{alg:stadholderate}) to find a common zero. 
  One potential common zero with $I(g) < 10^{-7}$ is given by 
  $$g_c = e^{-i \theta_1} (U_2(\theta_2, \theta_3, \theta_4) \otimes U_2(\theta_5, \theta_6, \theta_7)) U_4(\theta_8, \theta_9, \theta_{10}) (U_2(\theta_{11}, \theta_{12}, \theta_{13}) \otimes U_2(\theta_{14}, \theta_{15}, \theta_{16}))$$ where

  $$
    U_2(\alpha, \beta, \gamma) = 
    \begin{bmatrix}
      e^{-i \alpha} & 0 \\
      0 & e^{i \alpha}
    \end{bmatrix}
    \cdot 
    \begin{bmatrix}
      \cos \beta & -\sin \beta \\
      \sin \beta & \cos \beta
    \end{bmatrix}
    \cdot 
    \begin{bmatrix}
      e^{-i \gamma} & 0 \\
      0 & e^{i \gamma}
    \end{bmatrix},
  $$

  $$U_4(\alpha, \beta, \gamma) = \exp \lrpar{- I (\alpha \sigma_x \otimes \sigma_x + \beta \sigma_y \otimes \sigma_y + \gamma \sigma_z \otimes \sigma_z)}$$
  and 
  \begin{align*}
    \sigma_x = \begin{bmatrix}
      0 & 1 \\
      1 & 0
    \end{bmatrix}  
    , 
    \sigma_y = \begin{bmatrix}
      0 & -i \\
      i & 0
    \end{bmatrix}  
    , 
    \sigma_z = \begin{bmatrix}
      1 & 0 \\
      0 & -1
    \end{bmatrix}  
  \end{align*}
  with
  \begin{align*}
    (\theta_i)_{i=1}^{16} \approx  ( & 3.39082, 1.50097, 5.69898, 2.53181, 1.25383, 0.0170001, 6.21127, 
0.376407,\\
& 
 0.368786, 3.69014, 4.66335, 3.04854, 1.45524, 0.337423, 
3.38137, 3.82503)
  \end{align*}

                  \end{example}

\begin{remark}
  The $\mathcal{X}_2$-biinvariant function in $\rho_\lambda$ is the average of the value of the character $\chi_\lambda$ over the coset $g\mathcal{X}_2$. 
  Different from the situation of the $U(m)\times U(n-m)$-biinvariant function, we do not have a neat formula for the $\mathcal{X}_2$-biinvariant function. 
  If we regard $\chi_{(4,0,0,-4)(g)}$ as a function of the entries of $g$, then it is a polynomial of degree $4$ in the variables $g_{i,j}$, $1 \leq i,j \leq 4$ and of degree $4$ in the variables $\overline{g_{i,j}}$, $1 \leq i,j \leq 4$. 
  The number of monomials in the expansion is more than $10^4$. 
  Therefore the symbolic computation of the $\mathcal{X}_2$-biinvariant function is not easy. 
  On the other hand, the numerical computation of the value is relatively affordable.
\end{remark}

\section{Algorithms to find zeroes and common zeroes} \label{apdx:ahu}

\begin{algorithm}
  \caption{Find Zero by Bisection}\label{alg:ostensively}
  \begin{algorithmic}[1]
    \Function{FindZeroOnLinearGroup}{$f, L,R,\epsilon$}
      \While{$\norm{L-R} > \epsilon$}
        \State{$M \gets \exp \left((\log L + \log R)/2 \right)$.}
      \If{$f(M) < 0$}
        \State{$L \gets M$}
      \Else
        \State{$R \gets M$}
      \EndIf
      \EndWhile 
      \State{\Return{$M$}}
    \EndFunction
  \end{algorithmic}
\end{algorithm}

\begin{algorithm}
  \caption{Find Common Zero by Random Optimization}\label{alg:stadholderate}
  \begin{algorithmic}[1]
    \Function{FindCommonZero}{$I, \theta, \epsilon, h$}
      \State{$U \gets U(\theta)$}
      \While{$I(U) > \epsilon$}
        \State{$\Delta\theta \gets $ random parameters in $[0, 2\pi]$}
        \State{$U' \gets U(\theta + h\Delta\theta)$.}
      \If{$I(U') < I(U)$}
        \State{$\delta \gets \delta \frac{I(U')}{I(U)}$}
        \State{$\theta \gets \theta + h\Delta\theta$}
      \EndIf
      \EndWhile 
      \State{\Return{$U$}}
    \EndFunction
  \end{algorithmic}
\end{algorithm}

\end{document}